\documentclass[11pt]{amsart}
\usepackage{url, amssymb, enumerate, upref}
\newtheorem{theorem}{Theorem}[section]
\newtheorem{lemma}[theorem]{Lemma}
\newtheorem{proposition}[theorem]{Proposition}

\theoremstyle{definition}
\newtheorem{definition}[theorem]{Definition}
\newtheorem{example}[theorem]{Example}
\theoremstyle{remark}
\newtheorem{remark}[theorem]{Remark}
\numberwithin{equation}{section} 

\DeclareMathOperator{\diam}{diam}
\DeclareMathOperator{\disp}{disp}
\DeclareMathOperator{\med}{med}
\DeclareMathOperator{\dH}{d_H}

\title{Anti-absorbing ternary operations on metric spaces}

\author{Leonid V. Kovalev}
\address{215 Carnegie, Department of Mathematics, Syracuse University, Syracuse, NY 13244, USA}
\email{lvkovale@syr.edu}

\subjclass[2020]{Primary 51F30; Secondary 30L05, 54B20, 54C15} 

\keywords{Metric space, absolute retract, Lipschitz retract, ternary operation}

\begin{document}
\baselineskip6mm

\begin{abstract} The existence of a median-type ternary operation on a metric space is known to have a number of implications for the geometry of the space. For such operations, if two of the three arguments coincide, they also coincide with the output of the operation. We consider ternary operations with the opposite property: if two of the arguments coincide, the output is equal to the third one. The existence of such an operation is a necessary condition for the space to be an absolute retract.  
\end{abstract}

\maketitle

\section{Introduction} 

In 1979, van Mill and van de Vel~\cite{vanMill1} launched the investigation of topological spaces $X$ that admit a continuous ternary operation $\sigma$ with the \emph{absorption property}: 
\begin{equation}\label{eq-absorption}
\sigma(a, a, b)=\sigma(a, b, a)=\sigma(b, a, a)=a \quad \forall a, b\in X.    
\end{equation}
Such an operation, called a \emph{mixer} in~\cite{vanMill1}, is inherited by any retract of $X$, thus providing a strong necessary condition for a continuum to be an absolute retract~\cite{vanMill1, vanMill2}. The prototypical example of a mixer is the coordinate-wise median on the Hilbert cube $[0, 1]^\infty$. 

In this paper we compare~\eqref{eq-absorption} with its natural counterpart, the \emph{anti-absorption property}:  
\begin{equation}\label{eq-anti-absorption}
\tau(a, a, b)=\tau(a, b, a)=\tau(b, a, a) = b \quad \forall a, b\in X.    
\end{equation}
We call a continuous map 
$\tau\colon X^3\to X$ with the property~\eqref{eq-anti-absorption} a \emph{co-mixer}. The prototypical example of such an operation is the symmetric difference of sets: $\tau(A, B, C) = A\triangle B\triangle C$. 

When $X$ is a metric space, one often considers its Lipschitz retracts instead of merely continuous ones. Such retracts inherit Lipschitz (co-)mixers from the original space. 
Every absolute Lipschitz retract admits a Lipschitz mixer, and the converse holds under additional assumptions~\cite{Hohti}. 
In this paper we will show that: 
\begin{enumerate}[(i)]
\item a connected finite $CW$-complex with a co-mixer must be contractible (\S\ref{sec-topological})
\item a metric space with a Lipschitz co-mixer may have nontrivial fundamental group (\S\ref{sec-topological}) 
\item a metric space may admit a Lipschitz mixer without a Lipschitz co-mixer, and vice versa (\S\ref{sec-topological}--\ref{sec-metric})
\item normed vector spaces, as well as absolute Lipschitz retracts, admit a Lipschitz co-mixer (\S\ref{sec-mix-comix})
\end{enumerate}

As an application, in \S\ref{sec-retractions} we give a short proof of a theorem of Akofor~\cite{Akofor} on the Lipschitz retraction of finite subsets of normed spaces. 

\section{Topological considerations}\label{sec-topological}

We begin with a topological lemma involving the concept of an $H$-space: a topological space with a continuous binary operation with a two-sided unit. The book~\cite{Hatcher} contains the necessary background;  in particular, see~\cite[\S3.C]{Hatcher} on $H$-spaces.

\begin{proposition}\label{thm-homotopy} Suppose that $X$ is a topological space with a co-mixer $\tau$. Then for every point $e\in X$ and every $n\in \mathbb N$ the homotopy group $\pi_n(X, e)$ satisfies $2\alpha = 0$ for all $\alpha \in \pi_n(X, e)$. 
If, in addition, $X$ is a connected finite $CW$-complex, then it is contractible. 
\end{proposition}

\begin{proof} The binary operation $\mu(x, y) = \tau(x, y, e)$ satisfies $\mu(x, e) = x = \mu(e, x)$ for all $x\in X$ and thus makes $X$ an $H$-space. By the Eckmann-Hilton argument, the group operation on $\pi_n(X, e)$ is commutative and agrees with the one induced by $\mu$ (see Theorems~3.1 and~3.2 in~\cite{James}). Since $\mu(x, x) = e$ for all $x\in X$, it follows that $2\alpha = 0$ for all $\alpha\in \pi_n(X, e)$.
If, in addition, $X$ is a connected finite $CW$-complex, then the fact that $2\pi_n(X, e)=0$ for all $n$ implies that $X$ is contractible~\cite[\S3]{ArkowitzOshimaStrom}.     
\end{proof}

Every topological space with a mixer has trivial homotopy groups of all orders~\cite[Theorem 1.3]{vanMill1} (the theorem is stated for continua but the proof works in greater generality). In contrast, some spaces with a co-mixer are not simply-connected.  

\begin{proposition}\label{example-measure-algebra} There exists a connected metric space $X$ with a Lipschitz co-mixer and $\pi_1(X)\cong \mathbb Z/(2\mathbb Z)$. 
\end{proposition}

\begin{proof} Let $\lambda$ be the Lebesgue measure on $[0, 1]$. Let $\mathcal M$ be the associated measure algebra, i.e., the set of all measurable subsets of $[0, 1]$ modulo nullsets. When endowed with the metric $\rho(A, B)= \lambda(A\triangle B)$, the measure algebra becomes a contractible metric space. It admits both a Lipschitz mixer: 
\begin{equation}\label{eq-set-mixer}
\sigma(A, B, C) = 
(A\cap B)\cup (A\cap C)\cup (B\cap C) 
\end{equation}
and a Lipschitz co-mixer: 
\begin{equation}\label{eq-set-comixer}
\tau(A, B, C) = A\triangle B\triangle C.
\end{equation}

The set complement operation is an isometric involution on $\mathcal M$. By identifying each set with its complement we obtain a quotient space $(\widetilde{\mathcal M}, \widetilde{\rho})$ where 
\[
\widetilde{\rho}([A], [B]) = \min(\lambda(A\triangle B), \lambda(A\triangle B^c)). 
\]
By construction $\pi_1(\widetilde{\mathcal M}) \cong \mathbb Z/(2\mathbb Z)$, which implies that $\widetilde{\mathcal M}$ does not admit a mixer. In particular,~\eqref{eq-set-mixer} does not descend to the quotient. In contrast,~\eqref{eq-set-comixer} does: the induced co-mixer 
\[
\widetilde{\tau}([A], [B], [C]) = [A\triangle B\triangle C]
\]
is well-defined on $\widetilde{\mathcal M}^3$, and it is easy to see that $\widetilde{\tau}$ is $1$-Lipschitz in each argument. 
\end{proof}
 
In the context of nonsmooth metric spaces, the concept of Lipschitz homotopy is sometimes preferable to continuous homotopy~\cite{DHLT, HST}. It is thus relevant to note that the Lipschitz fundamental group of the space in Proposition~\ref{example-measure-algebra} is also nontrivial. Indeed, the curve $\gamma(t) = [0, t]$, $0\le t\le 1$, is a geodesic between $\varnothing$ and $[0, 1]$ in $\mathcal M$, and its image in the quotient space is a noncontractible rectifiable loop based at $\varnothing$. 

\section{Metric considerations}\label{sec-metric}

When $(X, d)$ is a metric space and $E\subset X$, we define the \emph{displacement} of a map $f\colon E\to X$ as $\disp(f) = \sup_{x\in E} d(f(x), x)$. A map $f\colon X\to Y$ is called $L$-Lipschitz if $d_Y(f(a), f(b))\le L d_X(a, b)$ for all $a,b\in X$.

\begin{lemma}\label{lem-interchange}
Suppose that $(X, d)$ is a metric space with a co-mixer $\tau$ that is $L$-Lipschitz in each variable. Then for any two points $a, b\in X$ there exists an $L$-Lipschitz map $f\colon X\to X$ such that $f(a)=b$, $f(b)=a$, and $\disp(f)\le Ld(a, b)$.    
\end{lemma}

\begin{proof} Let $f(x)=\tau(x, a, b)$. The displacement bound follows from the Lipschitz property of $\tau$ and the fact that $\tau(x, a, a)=x$ for all $x$. The other claimed properties of $f$ are also easy to see.       
\end{proof}

The map $f$ in Lemma~\ref{lem-interchange} need not have a Lipschitz inverse or even be injective; thus, this property is different from Lipschitz homogeneity~\cite{HerronMayer}. But it does imply some kind of homogeneity of the space, as the following lemma demonstrates. 

\begin{lemma}\label{example-unrectifiable}
Let $\Gamma$ be a metric arc with endpoints $a, b$; that is, $\Gamma$ is the image of a homeomorphism $\phi\colon [0, 1]\to \Gamma$ such that $\phi(0)=a$ and $\phi(1)=b$. Suppose that $\Gamma$ is unrectifiable but $\phi([\varepsilon, 1])$ is rectifiable for every $\varepsilon>0$. Then $\Gamma$ does not admit a Lipschitz co-mixer. 
\end{lemma}

\begin{proof} The image of a rectifiable arc under a Lipschitz map is also rectifiable. Thus, any Lipschitz map $f\colon \Gamma\to \Gamma$ with $f(b)=a$ must be identically equal to $a$. This shows that $\Gamma$ fails the conclusion of Lemma~\ref{lem-interchange}. 
\end{proof}

\begin{definition}\label{def-metric-BT} A metric space $(X, d)$ has \emph{bounded turning}  if there exists a constant $C$ such that any two points $a, b\in X$ are contained in some compact connected set $E\subset X$ with $\diam E \le Cd(a, b)$. 
\end{definition}

Corollary~3.3 in~\cite{Kovalev2023} shows that a metric arc has bounded turning if and only if it admits a Lipschitz mixer. On the other hand, Example~5.5 in~\cite{Kovalev2023} shows how to construct a metric arc $\Gamma$ of bounded turning which satisfies the assumption of Lemma~\ref{example-unrectifiable}. We thus obtain a metric arc with a Lipschitz mixer but no Lipschitz co-mixer. 

If one is willing to accept a non-connected space, then a simpler example of this kind is available. First, note that every subset of $\mathbb R$ admits a Lipschitz mixer, namely the median map (which we denote by $\med$).

\begin{example}\label{example-linear-gap} Let $X=\mathbb R\setminus (-1, 1)$ with the standard metric. Suppose that $f\colon X\to X$ is a Lipschitz map with $f(1)=-1$. By continuity, $f(x)\le -1$ for all $x\ge 1$. Hence $\disp(f)=\infty$, which by Lemma~\ref{lem-interchange} implies that $X$ has no Lipschitz co-mixer. 
\end{example}
 
It is not clear if there is a geometric description of the subsets of $\mathbb R$ that admit a Lipschitz co-mixer. For example, every additive subgroup of $\mathbb R$ does: let  
$\tau(a, b, c) = a+b+c-2\med(a, b, c)$.

\section{Ternary operations in normed spaces}\label{sec-mix-comix}

The authors of~\cite{vanMill2} asked ``whether every Banach space has a ``natural'' mixer'', thus contrasting their existence theorem~\cite[Theorem 2.4]{vanMill2} with the lack of a direct construction.  
Dranishnikov~\cite{Dranishnikov} sketched the following construction of a mixer in an arbitrary normed space, attributed to E. V. Shchepin. 

\begin{definition}\label{def-inradius-mixer} Given a normed space $X$, 
define the \emph{incenter mixer} $\sigma\colon X^3\to X$ by the formula
\begin{equation}\label{eq-mixer-def}
    \sigma(a, b, c)=\frac{\|b-c\|a + \|a-c\|b+\|a-b\|c}{\|b-c\| + \|a-c\|+\|a-b\|}, 
\end{equation}
extended by $\sigma(a, a, a)=a$. 
\end{definition}

When $X$ is a Euclidean space, $\sigma$ is the center of the inscribed circle in the triangle $abc$. It is the weighted average of $a, b, c$ where the weight of each vertex is the length of the opposite side. The mixer property of $\sigma$ is immediate from~\eqref{eq-mixer-def}. The Lipschitz continuity is less obvious, especially considering that the circumcenter fails this property.  

\begin{lemma}\label{lem-incenter-lip} In any normed space $X$, the incenter mixer~\eqref{eq-mixer-def} is $1$-Lipschitz in each variable separately. 
\end{lemma}

\begin{proof} By the translation invariance of $\sigma$, we can place one of the points at $0$, so that the mixer is computed for the triple $0, a, x$ with a fixed nonzero vector $a\in X$. Let
\begin{equation}\label{eq-F-sigma}
F(x) := \sigma(0, a, x) = \frac{\|a\|x + \|x\|a}{\|a\|+\|x\|+\|x-a\|}.
\end{equation}
By the homogeneity of $F$ it suffices to consider the case $\|a\|=1$. This simplifies~\eqref{eq-F-sigma} to 
\[
F(x) = \frac{x + \|x\|a}{p} \quad \text{where } p = 1+\|x\|+\|x-a\|.
\]
Since $p$ is bounded from below, the map $F$ is locally Lipschitz, hence absolutely continuous on lines. It remains to estimate its directional derivative $D_u F$ for a unit vector $u$. We have
\begin{equation}\label{eq-DuF-1}
\begin{split} 
p^2 D_u F(x) & = p (u + aD_u\|x\|) 
- (x+\|x\|a) (D_u\|x\|+D_u\|x-a\|) \\
&= pu  + \big(a-x+\|x-a\|a\big) D_u\|x\|  
- \big(x+\|x\|a\big)D_u\|x-a\|, 
\end{split}
\end{equation}
hence 
\[
p^2 \|D_u F(x)\| \le  p  + 2\|x-a\| + 2\|x\|.
\]
Since $p\ge 2$, it follows that
\[
p^2 = p + p(\|x-a\| + \|x\|) \ge p + 2\|x-a\| + 2\|x\|
\]
which proves that $\|D_uF(x)\|\le 1$.
\end{proof}

An alternative construction of a Lipschitz mixer is available in certain sequence spaces: take the component-wise median of three vectors $a, b, c$. In $\ell^\infty$ this construction produces a mixer with a stronger Lipschitz property than the one in  Lemma~\ref{lem-incenter-lip}: 
\begin{equation}\label{eq-joint-1-lip}
\|\med(a, b, c) - \med(a', b', c')\|   \le \max(\|a-a'\|, \|b-b\|, \|c-c'\|).  
\end{equation}
However, $\med(a, b, c)$ does not necessarily lie in the affine span of $a, b, c$ and therefore the median mixer is not inherited by subspaces. It remains unclear which normed spaces admit a mixer that is jointly nonexpanding in the sense of~\eqref{eq-joint-1-lip}.

\begin{definition}\label{def-nagel-comixer} Given a normed space $X$, 
define the \emph{Nagel co-mixer} 
$\tau\colon X^3\to X$ by the formula
\begin{equation}\label{eq-def-comixer}
\tau(a, b, c) = a+b+c - 2\sigma(a, b, c)    
\end{equation}
where $\sigma$ is as in~\eqref{eq-mixer-def}. 
\end{definition}
In the context of Euclidean spaces $\tau(a, b, c)$ is the {Nagel point} of triangle $abc$, i.e., the point at which three perimeter bisectors meet~\cite{Kimberling}. It can be expressed as the weighted average of $a, b, c$ where the weight of each vertex is the sum of two adjacent sides minus the opposite side. Both the co-mixer property and the Lipschitz continuity of $\tau$ (with constant $3$) follow from~\eqref{eq-def-comixer} at once. However, the Lipschitz constant can be improved. 

\begin{lemma}\label{lem-nagel-lip}
In any normed space $X$, the Nagel co-mixer~\eqref{eq-def-comixer} is $1$-Lipschitz in each variable separately. 
\end{lemma}

\begin{proof} As in the proof of Lemma~\ref{lem-incenter-lip}, it suffices to consider, for a fixed vector with $\|a\|=1$,  the map 
\[
G(x):=\tau(0, a, x) = a + x - 2 F(x) 
\]
where  
$F(x)=p^{-1} (x + \|x\|a)$ with $p = 1+\|x\|+\|x-a\|$.
Because of~\eqref{eq-DuF-1}, the directional derivative $D_uG$ is given by 
\[ 
p^2 D_u G(x) = (p^2 - 2p)u  - 2\big(a-x+\|x-a\|a\big) D_u\|x\|  
+2\big(x+\|x\|a\big)D_u\|x-a\| 
\]
Rearrange this as a linear combination of $u$, $x$, and $a$: 
\begin{equation}\label{eq-DuG-rearr}
\begin{split}
p^2 D_u G(x) & = (p^2 - 2p)u  + 2\big(D_u\|x\|+D_u\|x-a\|\big)x  
\\ &+ 2\big(\|x\|D_u\|x-a\|-(1+\|x-a\|)D_u\|x\| \big)a. 
\end{split}
\end{equation}
If the coefficient of $a$ in~\eqref{eq-DuG-rearr} is nonnegative, then the triangle inequality yields
\[
\begin{split}
p^2 \|D_u G(x)\| & \le p^2 - 2p
+ 2\big(D_u\|x\|+D_u\|x-a\|\big)\|x\| 
\\ & + 2 \|x\|D_u\|x-a\| - 2(1+\|x-a\|)D_u\|x\| 
\\ & \le p^2 - 2p
+ 4 \|x\|D_u\|x-a\| 
\le p^2 - 2p + 4\|x\| \le p^2
\end{split}
\]
since $p\ge 2\|x\|$. If the coefficient of $a$ in~\eqref{eq-DuG-rearr} is negative, then 
\[ \begin{split}
p^2 \|D_u G(x)\| & \le p^2 - 2p 
+2(D_u\|x\|)\|x\| +2(1+\|x-a\|)D_u\|x\| 
 \\ 
& \le p^2 - 2p   +2\|x\| 
+2(1+\|x-a\|) = p^2. 
\end{split}
\]
In either case we have $\|D_u G(x)\|\le 1$, completing the proof.
\end{proof}

\begin{remark}\label{remark-convexity}
Both the incenter mixer and the Nagel co-mixer of $a, b, c$ belong to the convex hull of the set $\{a, b, c\}$. 
\end{remark}

Every absolute Lipschitz retract $X$ can be realized as a retract of Banach space $\ell_\infty(I)$ for some index set $I$~\cite[Proposition 1.2]{BenyaminiLindenstrauss}. Lemma~\ref{lem-nagel-lip} implies that $X$ admits a Lipschitz co-mixer. The existence of a Lipschitz mixer on such spaces is well-known~\cite{HeinonenLip, Hohti}. 

\begin{remark} The ternary symmetric difference~\eqref{eq-set-comixer} is intertwined with the Nagel co-mixer~\eqref{eq-def-comixer} in the following way. The Lebesgue measure $\lambda$ is a map from $\mathcal M$ onto $[0, 1]$ with a right inverse  given by $\varphi(t) = [0, t]$. For any $a, b, c\in [0, 1]$ we have 
\[
\lambda\left(\varphi(a)\triangle \varphi(b)\triangle \varphi(c)\right) = a+b+c - 2\med(a, b, c).
\]
\end{remark}
 
\section{Retractions of finite subset spaces}\label{sec-retractions}

Any metric space $X$ can be viewed as the first member of an infinite sequence of nested metric spaces $(X(n), \dH)$, $n=1, 2, \dots$, where the elements of $X(n)$ are nonempty subsets of $X$ with at most $n$ elements, and $\dH$ is the Hausdorff metric. Given the natural isometric embeddings $X(n)\to X(n+1)$ one may ask whether Lipschitz retractions $X(n+1)\to X(n)$ also exist; this turns out to be a difficult question~\cite{Kovalev2022}. For example, the answer remains unknown when $X$ is a general normed space; the case of Hilbert spaces was settled in~\cite{Kovalev2016}. The first nontrivial case of this problem, $n=2$, has been solved by Akofor~\cite{Akofor} who constructed a $731$-Lipschitz retraction $X(3)\to X(2)$ for any normed space $X$. Here we present a simpler construction with a smaller Lipschitz constant. 

We say that a (co-)mixer is \emph{symmetric} if its value does not depend on the order of the three variables. Both the incenter mixer and the Nagel co-mixer are symmetric. The following lemma relates symmetric (co-)mixers to retractions.

\begin{lemma}\label{lem-retraction-from-mixing} 
Let $(X, d)$ be a metric space with a symmetric mixer $\sigma$ and a symmetric co-mixer $\tau$. If both $\sigma$ and $\tau$ are $L$-Lipschitz with respect to each argument, then
the map 
$\rho(\{a, b, c\}) = \{\sigma(a, b, c), \tau(a, b, c) \}$  
is a $9L$-Lipschitz retraction from $X(3)$ onto $X(2)$. \end{lemma} 

\begin{proof} A set $E\subset X$ with fewer than $3$ elements can be written as $\{a, b, c\}$ by listing the same point more than once. The definition of $\rho$ shows that regardless of which point was repeated, $\rho(E) = E$.  Given any two sets $A, B\in X(3)$, let $\delta=\dH(A, B)$ and pick two maps $f\colon A\to B$ and $g\colon B\to A$ such that $\disp(f)\le \delta$ and $\disp(g)\le \delta$. Define $h\colon B\to B$ so that $h(b)=b$ when $b\in f(A)$ and $h(b)=f(g(b))$ otherwise. Note that $h(B)=f(A)$ and $\disp(h)\le 2\delta$. 

Using the Lipschitz property of $\sigma$ and $\tau$ together with the displacement bounds for $f$ and $h$, we obtain $\dH(\rho(A), \rho(f(A))) \le 3L\delta$ and 
$\dH(\rho(B), \rho(h(B))) \le 6L\delta$.
By the triangle inequality,
$\dH(\rho(A), \rho(B))\le  9L\delta$ as claimed. 
\end{proof}

Combining Lemmas~\ref{lem-incenter-lip},  \ref{lem-nagel-lip} and~\ref{lem-retraction-from-mixing}, we obtain a sharper version of~\cite[Theorem~3.18]{Akofor}.

\begin{theorem}
Every normed space $X$ admits a $9$-Lipschitz retraction $X(3)\to X(2)$ such that the image of every set $A\in X(3)$ is contained in the convex hull of $A$. 
\end{theorem}

\bibliography{comixers-ref.bib} 
\bibliographystyle{plain} 
\end{document}